\documentclass{article}
\usepackage[latin1]{inputenc}

\usepackage{amsfonts}
\usepackage{amssymb}
\usepackage{amsthm}
\usepackage{amsmath}
\usepackage{mathrsfs}
\usepackage{stmaryrd}
\usepackage[latin1]{inputenc}
\usepackage[all]{xy}

\newcommand{\rac}{\mathbb Q}

\newcommand{\com}{\mathbb C}

\newcommand{\lin}{\textrm{lin}}
\newcommand{\zze}{\mathbb{Z}}

\newcommand{\zz}{\mathcal{Z}}
\newcommand{\lae}{\varepsilon}
\newcommand{\cp}{\mathcal P}

\newcommand{\bres}{\textrm{Res}}

\newcommand{\binf}{\textrm{Inf}}
\newcommand{\biso}{\textrm{Iso}}

\newcommand{\oviz}[1]{\overleftarrow{#1}}
\newcommand{\ovder}[1]{\overrightarrow{#1}}
\newcommand{\ind}[2]{\!\uparrow_{#1}^{#2}}

\theoremstyle{plain}
\newtheorem{teo}{Teorema}[section]
\newtheorem{prop}[teo]{Proposition}
\newtheorem{coro}[teo]{Corollary}
\newtheorem{lema}[teo]{Lemma}
\newtheorem{conj}[teo]{Conjecture}

\theoremstyle{definition}
\newtheorem{defi}[teo]{Definition}
\newtheorem{nota}[teo]{Notation}

\theoremstyle{remark}
\newtheorem{ejem}[teo]{Example}

\author{Nadia Romero\\
\begin{small}
\texttt{nadiaro$\, $@$\, $ciencias.unam.mx}\end{small}
}
\title{Simple modules over Green biset functors}
\date{ }

\begin{document}

\maketitle

\begin{abstract} 
 We present three examples of Green biset functors for which their simple modules can be parametrized. These are particular cases of a conjecture by Serge Bouc classifying the simple modules over a Green biset functor $A$, that generalizes the classification of simple biset functors. We also prove this conjecture under certain hypothesis for $A$.  
\end{abstract}

\section{Introduction}

 A Green functor for a finite group $G$ is defined as a Mackey functor $A$ with an additional multiplicative structure on each $A(H)$, for $H$ subgroup of $G$, compatible with the structure of Mackey functor. In the context of Mackey functors, Green functors have been extensively studied and many examples and applications of them have been found (see for example Thévenaz \cite{bthev}, Bouc \cite{bGfun} and Panchadcharam and Street \cite{elango}). In the context of categories of biset functors, it has been proved by Serge Bouc that if the class of objects in the biset category is closed under direct products, then the category of biset functors defined on it has a symmetric monoidal structure given by tensor product, the identity element is the Burnside functor. A Green biset functor $A$ is then a monoid in this category. %Both, the tensor product of biset functors and Green biset functors can be defined in different equivalent ways. 
That is, $A$ is a biset functor compatible with the monoidal structures of the biset category and that of $R$-Mod, when $R$ is a commutative ring with unity. This means that $A$ is equipped with bilinear products from $A(G)\times A(H)$ to $A(G\times H)$, for finite groups $G$ and $H$, which have a unit element and are associative and functorial in a natural sense. One feature of this practical definition is that it allows us to observe that many known Green functors are Green biset functors too.  

The aim of this article is towards the classification of simple Green biset functors, the results are part of my Ph.D. dissertation \cite{tesis}. As it happens with classical Green functors, the concept of module over a Green biset functor can be defined, and, as we will see in the next section, biset functors are modules over the Burnside functor. If $A$ is a Green biset functor, a \textit{left ideal} of $A$ is a submodule of the $A$-module $A$. Similarly one defines a \textit{right ideal}, and then a \textit{two sided ideal}. $A$ is called simple if its only two sided ideals are $\{0\}$ and $A$. One natural step is then, the classification of simple modules over Green biset functors.

In order to study the modules over a Green biset functor $A$, Bouc introduces a category $\cp_A$ such that $A$-modules correspond to $R$-linear functors from $\cp_A$ to $R$-Mod. 
The class of objects of $\cp_A$ is the same class of groups on which $A$ is defined, and if $G$ and $H$ are groups in $\cp_A$, then $Hom_{\cp_A}(G,\, H)=A(H\times G)$. Certain composition $\circ$ is then introduced for these arrows, it is described in the following section. It is a conjecture of Bouc that the simple modules over a Green biset functor $A$ are in correspondence with couples $(H,\, V)$ for which $H$ is a group such that the quotient algebra $\hat{A}(H)$ is different from zero and $V$ is a simple $\hat{A}(H)$-module. Here $\hat{A}(H)$ denotes the quotient of $A(H\times H)$ over the submodule generated by elements that can be factored through $\circ$ by groups strictly smaller than $H$. The classification of simple biset functors turns out to be a particular case of this conjecture. We present here other three examples which satisfy it and we prove it under the assumption that minimal groups for simple $A$-modules are unique up to isomorphism. The examples are:
 
\begin{enumerate}
\item The functor of rational representations with coefficients in a field of characteristic zero, $kR_{\rac}$. Modules over $kR_{\rac}$ are also known as rhetorical biset functors and simple modules were already classified by Laurence Barker \cite{barker}. We present a different proof of this classification, using the fact that $kR_{\rac}$ is an homomorphic image of the Burnside functor $kB$. 

\item The functor of complex representations with coefficients in the complex field,  $\com R_{\com}$. In this case, $\com R_{\com}$ is the only simple $\com R_{\com}$-module, and so it is a simple Green biset functor.

\item The Yoneda-Dress construction at a group $C$ of prime order of the Burnside functor, $RB_{C}$. 
\end{enumerate}

\section{Notation and settings}
By a group we will always mean a finite group. If $G$ is a group, by a $G$-set, we will always mean a finite $G$-set.

Let $\zz$ be a class of groups. We will say that $\zz$ is closed under subquotients if given a group $G$ in $\zz$, then  the image of any morphism of groups from a subgroup  of $G$ belongs to $\zz$. We will say that $\zz$ is closed under direct products if given $G$ and $H$ in $\zz$, their direct product $G\times H$ is in $\zz$.

The definitions appearing in this section, as well as more results and theory around them, can be found in Bouc \cite{biset}.

We will write $B(H,\, G)$ for the Grothendieck group of the category of finite $(H,\, G)$-bisets. If $R$ is a commutative ring with identity, we will write $RB(H,\, G)$ for $R\otimes_{\zze} B(H,\, G)$. Given $U$ an $(H,\, G)$-biset and $V$ a $(K,\, H)$-biset, the composition of $V$ and $U$, denoted by $V\times_H U$ is the set of $H$-orbits on the cartesian product $V\times U$, where the right action of $H$ is defined by
\begin{displaymath}
 \forall (v,\, u)\in V\times U,\, \forall h\in H,\ (v,\, u)\cdot h=(v\cdot h,\, h^{-1}\cdot u). 
\end{displaymath}
This composition has a natural structure of $(K,\, G)$-biset.

\begin{nota}
Let $\zz$ be a class of groups closed under subquotients. We will write $\Omega_{R,\, \zz}$ for the biset category with objects in $\zz$. This means that if $G$ and $H$ are in $\zz$, then $Hom_{\Omega_{R,\, \zz}}(G,\, H)=RB(H,\, G)$. The composition of morphism is
given by the map
\begin{displaymath}
\_\, \times_H\,\_:RB(K,\, H)\times RB(H,\, G)\longrightarrow RB(K,\, G).
\end{displaymath}
The identity element in $RB(G,\, G)$ is the $(G,\, G)$-biset $G$.

If $\zz$ is the class of all finite groups, we will write $\Omega_R$ instead of $\Omega_{R,\, \zz}$.
\end{nota}

If $H$ and $G$ are groups and $L\leqslant H\times G$, then the corresponding element in $RB(H,\, G)$ satisfies the Bouc decomposition (2.3.26 in \cite{biset}):
\begin{displaymath}
 \textrm{Ind}_D^H\times_D \textrm{Inf}_{D/C}^D\times_{D/C} \textrm{Iso}(f)\times_{B/A} \textrm{Def}_{B/A}^B\times_B \textrm{Res}_B^G
\end{displaymath}
with $C\trianglelefteqslant D\leqslant H$, $A\trianglelefteqslant B\leqslant G$ and $f:B/A\rightarrow D/C$ a group isomorphism. 

The biset functors we will consider are $R$-linear functors from a category $\Omega_{R,\, \zz}$ to the category $R$-Mod. The Burnside  functor defined in $\Omega_{R,\, \zz}$ is denoted by $RB$. In the case of Green biset functors, we will always assume that $\zz$ is closed under direct products. 

\begin{nota}
Let $G$, $G'$, $H$ and $H'$ be groups. If $U$ is an $(H,\, G)$-biset and $U'$ is an $(H',\, G')$-biset, then the cartesian product $U\times V$ has a natural structure of 
$(H\times H',\, G\times G')$-biset. By linearity, this construction yields a map 
\begin{displaymath}
RB(H,\,G)\times RB(H',\, G')\rightarrow RB(H\times H',\, G\times G'),
\end{displaymath}
also denoted by $(\alpha,\, \beta)\mapsto \alpha\times \beta$.

If $\alpha$ is any element in $RB(H,\, G)$ and $K$ is a group, $\alpha \times K$ will refer to this construction applied to $\alpha$ and the $(K,\, K)$-biset $K$.
\end{nota}

The following definition of Green biset functor is 8.5.1 in \cite{biset}, it is given in terms of the tensor product of bisets functors, defined in 8.4.1 of the same reference. 

\begin{defi}
\label{defgten}
Let $\zz$ be a class of groups closed under subquotients and direct products. A biset functor $A$ defined in $\Omega_{R,\, \zz}$ is a Green biset functor if there exist maps of biset functors
\begin{displaymath}
 \mu:A\otimes A\rightarrow A,\quad \textrm{y}\quad e:RB\rightarrow A
\end{displaymath}
such that the following diagrams commute
\[
 \xymatrix @R=.1in @C=.3in {
A\otimes (A\otimes A)\ar[dd]_{\cong}\ar[r]^-{Id\otimes \mu} & A\otimes A\ar[rd]^-{\mu}\\
 & & A\quad\\
(A\otimes A)\otimes A\ar[r]^-{\mu\otimes Id} & A\otimes A\ar[ru]_-{\mu}
}\quad
\xymatrix @R=.4in @C=.3in {
RB\otimes A\ar[r]^-{e\otimes Id}\ar[rd]_-{\cong} & A\otimes A\ar[d]^-{\mu} & A\otimes RB\ar[l]_-{Id\otimes e}\ar[ld]^-{\cong}\\
& A & \quad \quad\quad . 
} 
\]
\end{defi}

It is shown in Section 8.5 of \cite{biset} that this definition is equivalent to the following. 

\begin{defi}
\label{defgreen}
$A$ is a Green biset functor if it is a biset functor quipped with
bilinear products $A(G)\times A(H)\rightarrow A(G\times H)$ denoted by $(a,\, b)\mapsto a\times b$, for groups $G,\, H\in \zz$, and an element $\lae_A\in A(1)$, satisfying the following conditions:
\begin{itemize}
 \item Associativity. Let $G$, $H$ and $K$ be groups in $\zz$. If
 $\alpha_{G,\, H,\, K}$ is the canonical group isomorphism from $G\times (H\times K)$ to $(G\times H)\times K$, then for any $a\in A(G)$, $b\in A(H)$ and $c\in A(K)$
\begin{displaymath}
 (a\times b)\times c=A(\textrm{Iso}(\alpha_{G,\, H,\, K}))(a\times (b\times c)).
\end{displaymath}
\item Identity element. Let $G$ be a group in $\zz$. Let $\lambda_G:1\times G\rightarrow G$  and $\rho_G: G\times 1 \rightarrow G$ denote the canonical group isomorphisms. Then for any $a\in A(G)$
\begin{displaymath}
 a=A(\textrm{Iso}(\lambda_G))(\lae_A\times a)=A(\textrm{Iso}(\rho_G))(a\times \lae_A).
\end{displaymath}
\item Functoriality. If $\varphi: G\rightarrow G'$ and $\psi: H\rightarrow H'$ are morphisms in $\Omega_{R,\, \zz}$, then for any $a\in A(G)$ and $b\in A(H)$
\begin{displaymath}
 A(\varphi \times \psi)(a\times b)=A(\varphi)(a)\times A(\psi)(b).
\end{displaymath}
 \end{itemize}
\end{defi}

\begin{ejem}
 Assigning to a group $G$ the Burnside ring $B(G)$ defines a biset functor in $\Omega_{\mathbb{Z}}$.  
In order to define the products $B(G)\times B(H)\rightarrow B(G\times H)$, let $a$ be in $B(G)$ and $b$ be in $B(H)$. Then $B(\textrm{Inf}_G^{\,G\times H})(a)$ is in $B(G\times H)$ and 
$B(\textrm{Inf}_H^{\,G\times H})(b)$ is also in $B(G\times H)$, and so $a\times b$ is given by the natural multiplication of these two elements in $B(G\times H)$. In other words, this product is induced by the bifunctor sending a $G$-set $X$ and an $H$-set $Y$ to the $(G\times H)$-set $X\times Y$. 
\end{ejem}

\begin{ejem}
Let $\mathbb{F}$ be a field of characteristic zero and $\zz$ the class of finite groups. As in the previous example,  assigning to each group $G$ in $\zz$ the Grothendieck group of finitely generated $\mathbb{F}G$-modules, $R_{\mathbb{F}}(G)$, defines a biset functor $R_{\mathbb{F}}$. It also has a structure of Green biset functor given in the following way: Let $s\in R_{\mathbb{F}}(G)$ and $t\in R_{\mathbb{F}}(K)$, then $R_{\mathbb{F}}(\textrm{Inf}_G^{\,G\times K})(s)$  and $R_{\mathbb{F}}(\textrm{Inf}_K^{\,G\times K})(t)$ are both in $R_{\mathbb{F}}(G\times K)$ and so the natural product in $R_{\mathbb{F}}(G\times K)$, given by $\otimes_{\mathbb{F}}$, gives us the product $s\times t$.
\end{ejem}

These examples show us the way to pass from many known Green functors to Green biset functors.
 Let $\zz$ be closed under subquotients and direct products. Suppose $A$ is a biset functor defined in $\Omega_{R,\, \zz}$ such that $A(H)$ is an $R$-algebra with unity for each $H\in \zz$, and which satisfies the following: Let $f:K\rightarrow L$ be a group homomorphism for $K,\, L\in \zz$, write $X$ for the natural $(K,\, L)$-biset $L$ and $Y$ for the $(L,\, K)$-biset $L$. Then $A(X)$ is a morphism of unital $R$-algebras and $A$ satisfies the Frobenius reciprocity relations, that is, for all $b\in A(L)$ and $a\in A(K)$
\begin{displaymath}
A(Y)(a)\cdot b=A(a\cdot A(X)(b))\quad \textrm{and}\quad b\cdot A(Y)(a)=A(Y)(A(X)(b)\cdot a).
\end{displaymath}
In this case, $A$ has also a structure of Green biset functor. To define the products $A(G)\times A(H)\rightarrow A(G\times H)$, let $a\in A(G)$ and $b\in A(H)$. We take $A(\textrm{Inf}_G^{\,G\times H})(a)$ and $A(\textrm{Inf}_H^{\,G\times H})(b)$, make their product in $A(G\times H)$ and define $a\times b$ as this product. Then it can be shown that $A$ satisfies the three conditions of Definition \ref{defgreen}, with $\lae_A$ being the identity element of the original product in $A(1)$.

\begin{defi}
 Let $A$ and $C$ be Green biset functors defined in $\Omega_{R,\, \zz}$. A morphism of Green biset functors from $A$ to $C$ is a morphism of biset functors $f:A\rightarrow C$ such that $f_G(a)\times f_H(b)=f_{G\times H}(a\times b)$ for any groups $G,\, H\in \zz$ and any $a\in A(G)$, $b\in A(H)$.
\end{defi}

\begin{ejem}
For a field $\mathbb{F}$ of characteristic $0$, the linearization morphism
\begin{displaymath}
 \textrm{lin}_{\mathbb{F},\, G}:B(G)\rightarrow R_{\mathbb{F}}(G)
\end{displaymath}
sending each $G$-set $X$ to the $\mathbb{F}G$-module $\mathbb{F}X$ is a morphism of biset functors (Remark 1.2.3 in \cite{biset}). By the two previous examples, clearly it is a morphism of Green biset functors.
\end{ejem}

\begin{defi}
\label{defmova}
 Let $\zz$ be a class of groups closed under subquotients and direct products and $A$ be a Green biset functor defined in $\Omega_{R,\, \zz}$. An $A$-module is a biset functor $M$ from  $\Omega_{R,\, \zz}$ to $R$-Mod together with morphism $A(G)\times M(H)\rightarrow M(G\times H)$ for groups $G,\, H\in \zz$, denoted by $(a,\, m)\mapsto a\times m$, fulfilling the following conditions.
\begin{itemize}
 \item Associativity. Let $G$, $H$ and $K$ be groups in $\zz$. If
 $\alpha_{G,\, H,\, K}$ is the canonical group isomorphism from $G\times (H\times K)$ to $(G\times H)\times K$, then for any $a\in A(G)$, $b\in A(H)$ and $m\in M(K)$
\begin{displaymath}
 (a\times b)\times m=M(\textrm{Iso}(\alpha_{G,\, H,\, K}))(a\times (b\times m)).
\end{displaymath}
\item Identity element. Let $G$ be a group in $\zz$. Let $\lambda_G:1\times G\rightarrow G$  denote the canonical group isomorphism. Then for any $m\in M(G)$
\begin{displaymath}
 m=M(\textrm{Iso}(\lambda_G))(\lae_A\times m)
\end{displaymath}
\item Functoriality. If $\varphi: G\rightarrow G'$ and $\psi: H\rightarrow H'$ are morphisms in $\Omega_{R,\, \zz}$, then for any $a\in A(G)$ and $m\in A(H)$
\begin{displaymath}
 M(\varphi \times \psi)(a\times b)=A(\varphi)(a)\times M(\psi)(m).
\end{displaymath}
 \end{itemize}
\end{defi}

As it is seen in Section 8.5 of \cite{biset}, $A$-modules form a category, denoted by $A$-Mod.

\subsection*{The category associated to a Green biset functor}

The following proposition is point 5 in Proposition 8.6.1 of \cite{biset}, where it appears without proof. 

\begin{nota}
If $X$ is a $(G,\, H)$-biset, let $\ovder{X}$ be the $(G\times H,\, 1)$-biset $X$ with action $(g,\, h)x=gxh^{-1}$ for $g$ in $G$, $h$ in $H$ and $x$ in $X$. %%Denote by $\oviz{X}$ the $(1,\, G\times H)$-biset $(\ovder{X})^{op}$. 
This construction induces a map $\alpha \mapsto \ovder{\alpha}$ from $B(G,\, H)$ to $B(G\times H,\, 1)$. 
For $G$ a group,  $\ovder{G}$ refers to this notation applied to the $(G,\, G)$-biset $G$. In this case, $\oviz{G}$ will denote the $(1,\, G\times G)$-biset $\ovder{G}^{op}$.
\end{nota}

\begin{prop}
\label{catmod}
 Let $A$ be a Green biset functor over $\Omega_{R,\, \zz}$ and let $\cp_A$ be the following category:
\begin{itemize}
 \item The objects of $\cp_A$ are the groups in $\zz$.
 \item If $G$ and $H$ are in $\zz$, then $Hom_{\cp_A}(H,\, G)=A(G\times H)$.
 \item Let $H,\, G$ and $K$ be groups in $\zz$. The composition of $\beta\in A(H\times G)$ and $\alpha\in A(G\times K)$ in $\cp_A$ is the following:
\begin{displaymath}
\beta \circ \alpha = A(H\times \oviz{G}\times K)(\beta\times\alpha).
\end{displaymath}
\item If $G$ is in $\zz$, then the identity morphism of $G$ in $\cp_A$ is $A(\ovder{G})(\varepsilon_A)$.
\end{itemize}
Then $\cp_A$ is an $R$-linear category and $A$-Mod is equivalent to the category of  $R$-linear functors from $\cp_A$ in $R$-Mod.
\end{prop}
We will not present the complete proof here, we will only give the maps of the equivalence, leaving to the reader the details of proving it certainly is. 
\begin{proof}
We will write Fun$(\cp_A\rightarrow R\textrm{-Mod})$ for the category of functors from $\cp_A$ to $R$-Mod. 

First we define a functor
\begin{displaymath}
S: A\textrm{-Mod}\longrightarrow \textrm{Fun}(\cp_A\rightarrow R\textrm{-Mod}).
\end{displaymath}
Let $M$ be an $A$-module, define $F_M:\cp_A\rightarrow R$-Mod in a group $G$ as $M(G)$. For a morphism $\alpha\in A(G\times H)$, define
\begin{displaymath}
F_M(\alpha): M(H)\rightarrow M(G)\quad m\mapsto M(G\times\oviz{H})(\alpha\times m).
\end{displaymath}
Once proved that $F_M$ is a functor, it is not hard to see that if
 $f:M\rightarrow N$ is an arrow of $A$-modules, then $f$ defines a morphism from $F_M$ to $F_N$.

Now let $F$ be a functor from $\cp_A$ to $R$-Mod. Then it is a biset functor through the morphism $e: RB\rightarrow A$ of Definition \ref{defgten}. That is, if $\alpha$ is a $(G, H)$-biset, then $F(e_{G\times H}(\alpha))$ is a morphism of $R$-modules from $F(H)$ to $F(G)$.  
Now, in proving the equivalence between Definitions \ref{defgten} and \ref{defgreen}, one observes that  $e_{G\times H}(\alpha)=A(\ovder{\alpha})(\lae_A)$. With this, it is easy to see that $e$ is a morphism of Green biset functors, and then, that $F$ is functorial in $\Omega_{R,\, \zz}$, using the lemma stated right after this proof. 
Also, if $X_H=(H\times H)/\Delta(H)$, we have that $e_{H\times H}(X_H)$ equals $A(\ovder{H})(\lae_A)$, and so $F(X_H)(m)=m$.

To make $F$ an $A$-module, define the product
\begin{displaymath}
A(G)\times F(H)\rightarrow F(G\times H)\quad (a,\, m)\mapsto F(A(G\times \ovder{H})(a))(m).
\end{displaymath}
After proving this product satisfies the three conditions of Definition \ref{defmova}, we have the map in objects of the functor
\begin{displaymath}
T: \textrm{Fun}(\cp_A\rightarrow R\textrm{-Mod})\longrightarrow A\textrm{-Mod}, 
\end{displaymath}
sending $F$ to itself. Now, if $t:F\rightarrow E$ is an arrow in the category of the left hand side, because of the way $F$ was defined in bisets, it is clear that $t$ is also an arrow of biset functors. On the other hand, $t$ commutes with morphisms in $\cp_A$, this implies that $t$ behaves well with the product defined above. 
\end{proof}

The following lemma will be used throughout the rest of the article. We leave the easy proof to the reader.

\begin{lema}
\label{mult}
 Let $A$ and $C$ be Green functors and $f:C\rightarrow A$ a morphism of Green functors. If $\beta\in C(H\times G)$ and $\alpha\in C(G\times K)$, then
\begin{displaymath}
 f_{H\times K}(\beta\circ\alpha)=f_{H\times G}(\beta)\circ f_{G\times K}(\alpha). 
\end{displaymath}
\end{lema}

As a first application of the previous proposition, we have the following example.

\begin{ejem}
\label{bovbur}
Every biset functor is an $RB$-module. First, observe that the composition in $\cp_{RB}$ coincides with the known composition for bisets. That is, if $\beta$ is an $(H,\, G)$-biset and $\alpha$ is a $(G,\, K)$-biset, there is an isomorphism of $(H,\, K)$-bisets between
\begin{displaymath}
 RB(H\times \oviz{G}\times K)(\beta\times\alpha)\quad \textrm{and}\quad \ \beta\times_G \alpha.
\end{displaymath}
The Bouc decomposition for $\oviz{G}$ is
\begin{displaymath}
\textrm{Def}_{1}^{\Delta (G)}\times_{\Delta (G)} \textrm{Res}_{\Delta (G)}^{G\times G},
%\bic{1}{1}{\Delta (G)}\circ \bic{\Delta (G)}{G\times G}{G\times G}. 
\end{displaymath}
so $RB(H\times \oviz{G}\times K)(\beta\times\alpha)$ is the deflation from $H\times \Delta (G)\times K$ to $H\times K$ of $\beta\times \alpha$ having the following action as $H\times \Delta (G)\times K$-set 
\begin{displaymath}
(h,\, g,\, g,\, k)\cdot (x,\, y)=(hxg^{-1},\, gyk^{-1}).  
\end{displaymath}
This deflation is then the biggest quotient of $\beta\times\alpha$ in which $1\times \Delta (G)\times 1$ acts in the trivial way, and this is precisely $\beta\times_G \alpha$. 

So, the category $\cp_{RB}$ is equivalent to $\Omega_{R,\, \zz}$, and any biset functor is an $RB$-module.
\end{ejem}

A deep study of simple biset functors can be found in Chapter 4 of \cite{biset}. There we can find that the isomorphism classes of simple $RB$-modules are indexed by isomorphism classes of seeds $(H,\, V)$, where $H$ is a group in $\zz$ and $V$ is a simple $ROut(H)$-module. To each seed $(H,\, V)$ corresponds the isomorphism class of $S_{H,\, V}$, where $S_{H,\, V}(G)$ is the quotient of 
$L_{H,\, V}(G)=RB(G, H)\otimes_{RB(H,\, H)}V$ over
\begin{displaymath}
 J_{H,\, V}(G)=\Big\{\sum_{i=1}^n\varphi_i\otimes n_i\in L_{H,\, V}(G) \mid %\varphi_i\in A_R(G,\, H),\ n_i\in V\ \textrm{y}\
 \sum_{i=1}^n(\psi\circ \varphi_i)\cdot n_i=0\ \forall \psi \in RB(H, G) \Big\}.
\end{displaymath}

An important characteristic of $H$ is that it is a minimal group for $S_{H,\, V}$, and for a simple biset functor all its minimal groups are isomorphic.

\begin{defi}
 \label{defmin}
Let $M$ be a biset functor defined on $\Omega_{R,\, \zz}$. A group $H$ in $\zz$ is called minimal for $M$ if $M(H)\neq 0$ and $M(K)=0$ for any $K$ in $\zz$ with $|K|<|H|$.
\end{defi}

This motivates the following definition.

\begin{defi}
 \label{defalgq}
Let $A$ be a Green biset functor on $\Omega_{R,\, \zz}$, and $H$ be a group in $\zz$. We will write $I_A(H)$ for the submodule of $A(H\times H)$ generated by elements of the form $a\circ b$, where $a$ is in $A(H\times K)$, $b$ is in $A(K\times H)$ and $K$ is a group in $\zz$ of order smaller than $|H|$. We will denote by $\hat{A}(H)$ the quotient $A(H\times H)/I_A(H)$.
\end{defi}

It is clear that $\hat{A}(H)$ is an $R$-algebra.  If $A=RB$, it is known (Proposition 4.3.2 in \cite{biset}, for instance) that $\hat{RB}(H)$ is isomorphic to $ROut(H)$. Nonetheless, we will see in the following section that this quotient $\hat{RB}(H)$ may vanish.

The classification of simple biset functors is a particular case of the following conjecture.

\begin{conj}[S. Bouc, personal communication]
\label{laconj}
Suppose that $A$ is a Green biset functor. Let $\mathcal{S}_A$ be the set of equivalence classes of couples $(H,\, V)$, where $\hat{A}(H)\neq 0$, $V$ is a simple $\hat{A}(H)$-module and $(H,\, V)\sim (G,\, W)$ if there exists an isomorphism of groups $\varphi:H\rightarrow G$ such that $V\cong {}^{\varphi}W$. 

Then the isomorphism classes of simple $A$-modules are in one-to-one correspondence with the elements of $\mathcal{S}_A$.
\end{conj}

The details of the isomorphism $V\cong {}^{\varphi}W$ are given in Section 4. 

In the following two sections we will present other three examples of Green biset functors which also satisfy this conjecture.

If $S$ is a simple $A$-module, then $S\neq 0$, so taking $H$ as a group of minimal order such that $S(H)\neq 0$, we have a minimal group for $S$.

\begin{lema}
\label{minis}
Let $S$ be a simple $A$-module for a Green biset functor $A$ defined in $\Omega_{R,\,\zz}$. If $H\in \zz$ is a minimal group for $S$, then $\hat{A}(H)\neq 0$ and $S(H)$ is a simple $\hat{A}(H)$-module.
\end{lema}
\begin{proof}
 Suppose we have $\hat{A}(H)=0$, then in particular $A(\ovder{H})(\lae_A)=\sum_{i=1}^r\alpha_i\circ\beta_i$ where $\alpha_i\in A(H\times K_i)$ and $\beta_i\in A(K_i\times H)$ with $K_i$ a group of order smaller than $|H|$. But $S(\alpha_i\circ\beta_i)=0$ because $H$ is minimal for $S$, then $S(A(\ovder{H})(\lae_A))=0$, a contradiction. 

Clearly, $S(H)$ is an $\hat{A}(H)$-module, so it suffices to see it is a simple $A(H\times H)$-module. Let $W$ be an $A(H\times H)$-submodule, define 
\begin{displaymath}
 T(G)=\{m\in S(G)\mid \forall X\in A(H\times G),\ S(X)(m)\in W\}.
\end{displaymath}
It is easy to see that $T$ is a subfunctor of $S$. Besides, $T(H)=W$. Now, if $t\in T(H)$, then in particular for $X=A(\ovder{H})(\lae_A)$ we must have $T(X)(t)\in W$, then $T(H)\subseteq W$. On the other hand, since $W$ is a $A(H\times H)$-submodule, then $S(X)(w)$ is in $W$ for all $w\in W$ and $X\in A(H\times H)$, and so $W\subseteq T(H)$.

Hence, if $T=0$ we have $W=0$, and if $T=S$ then $W=S(H)$.
\end{proof}

\section{$k R_{\rac}$-modules or rhetorical biset functors}

In this section $k$ will be a field of characteristic $0$. 

Rhetorical biset functors were introduced by Laurence Barker in \cite{barker}, there he proves that every rhetorical biset functor over $k$ is semisimple and gives a classification of the simple functors. 
The definition of rhetorical biset functor we will give next is not exactly the one appearing in \cite{barker} but, under the hypothesis we will make, it is equivalent to that one. %$\mathbb{F}$ will be a field of characteristic $0$. 

Denote by $\hat{\Omega}_k$ the category having the class of finite groups as objects, and where the morphisms from $G$ to $H$ is the  quotient $k B(H\times G) / \textrm{}Ker(k\textrm{lin}_{\mathbb{F}\,H\times G})$, with $\mathbb{F}$ a field of characteristic 0 and 
\begin{displaymath}
 k\lin_{\mathbb{F}}:kB\rightarrow kR_{\mathbb{F}}
\end{displaymath}
being the linearization morphism. Thanks to Lemma \ref{mult}, the composition $\circ$ can be extended to this quotient. A rhetorical biset  functor is defined as a  $k$-linear functor from  $\hat{\Omega}_{k}$ to $k$-Mod.

Proposition \ref{catmod} and Artin's induction theorem gives us that in the case $\mathbb{F}=\rac$, rhetorical biset functors coincide with $kR_{\rac}$-modules. An equivalent statement is made in Proposition 3.1 of $\cite{barker}$, in view of the following lemma, which describes the composition $\circ$ in the category $\mathcal{P}_{kR_{\rac}}$.

\begin{lema}
 \label{prodmod}
Let $H,\, G$ and $K$ be groups. If $[M]$ is in $kR_{\rac}(H\times G)$ and $[N]$ is in $kR_{\rac}(G\times K)$, then
\begin{displaymath}
kR_{\rac}(H\times\oviz{G}\times K)([M]\times [N])=[M\otimes_{\rac G} N]. 
\end{displaymath} 
\end{lema}
\begin{proof}
 Recall that the product $[M]\times [N]$ is defined as the isomorphism class 
of the $\rac (H\times G\times G\times K)$-module $M\otimes_{\rac} N$, where  $G\times K$ acts in the trivial way in $M$ and  $H\times G$ acts in the trivial way in $N$. 
From the Bouc decomposition for $\oviz{G}$, we have that applying $A(H\times\oviz{G}\times K)$ to this module  gives the deflation from 
$H\times \Delta(G)\times K$ to $H\times K$ of the module $M\otimes_{\rac} N$ restricted to $H\times \Delta(G)\times K$.
But this deflation is the biggest quotient of $M\otimes_{\rac} N$ in which $1\times \Delta(G)\times 1$ acts in the trivial way, that is $M\otimes_{\rac G} N$.
\end{proof}

Hence, Theorem 1.5 of \cite{barker} gives a classification of simple $kR_{\rac}$-modules: They are the simple biset functors indexed by seeds $(H,\, V)$ where $H$ is a cyclic group of order $m$ and $V$ is a primitive $k(\zze/m\zze)^{\times}$-module. 

\begin{defi}
 Let $m$ be in $\mathbb N-\{0\}$. A simple $k (\zze/m\zze)^{\times}$-module $V$ is called  primitive if given $n$ a divisor of $m$ and $\pi_{m,\, n}:(\zze/m\zze)^{\times}\rightarrow (\zze/n\zze)^{\times}$ the natural projection, if $Ker\pi_{m,\, n}$ acts trivially on $V$, then $n=m$.
\end{defi}

The rest of this section is devoted to give a different proof of this classification, using the techniques of Green biset functors. First, we observe that this classification is a particular case of  Conjecture \ref{laconj}. This fact can also be concluded from Proposition \ref{isosim}, by proving (as we will do  with $\com R_{\com}$ and $RB_{C}$) that for each simple $kR_{\rac}$-module its minimal groups are unique up to isomorphism. On the other hand, using the fact that $k\lin_{\rac}$ is an epimorphism of Green biset functors, we can prove directly not only that $kR_{\rac}$ satisfies the parametrization of Conjecture \ref{laconj}, but also that simple $kR_{\rac}$-modules are simple biset functors. 

In proving the following lemma, we will use the fact, easy to proof, that it $f:A\rightarrow C$ is an  epimorphism of Green biset functors, then the simple $C$-modules are the simple $A$-modules in which $Kerf$ acts in the trivial way. 

\begin{lema}
 The simple $kR_{\rac}$-modules are the simple biset functors $S_{H, V}$, where $H$ satisfies $\hat{kR}_{\rac}(H)\neq 0$ and $V$ is a simple $\hat{kR}_{\rac}(H)$-module.
\end{lema}
\begin{proof}
 Since the linearization morphism $k$lin$_{\rac}$ is an epimorphism of Green biset functors, the simple $kR_{\rac}$-modules are the simple biset functors $S_{H,\, V}$ in which $Ker( k$lin$_{\rac})$ acts trivially.  Hence, by Lemma \ref{minis}, we have that $\hat{kR}_{\rac}(H)\neq 0$ and $V$ is a simple $\hat{kR}_{\rac}(H)$-module.

Now let $S_{H,\, V}$ be a simple biset functor  such that $\hat{kR}_{\rac}(H)\neq 0$ and $V$ is a simple $\hat{kR}_{\rac}(H)$-module. 

Take $K$ and $G$ arbitrary groups, %then by Section 4 in Bouc \cite{ensemble}, 
as we said in the previous section, $S_{H,\, V}(G)$ is the quotient of $L_{H,\, V}(G)=kB(G,\, H)\otimes_{kB(H,\, H)}V$ over its unique maximal subfunctor 
\begin{displaymath}
 J_{H,\, V}(G)=\Big\{\sum_{i=1}^n\varphi_i\otimes n_i\in L_{H,\, V}(G) \mid %\varphi_i\in A_R(G,\, H),\ n_i\in V\ \textrm{y}\
 \sum_{i=1}^n(\psi\circ \varphi_i)\cdot n_i=0\ \forall \psi \in kB(H,\, G) \Big\}.
\end{displaymath}
Let $a$ be in $Ker (k$lin$_{\rac\, K})$, we shall prove that $a$ acts in the trivial way on $S_{H,\, V}(G)$. It suffices then, to prove that if $m=\sum_{i=1}^n$\mbox{$\varphi_i\otimes n_i$} is an element in $L_{H,\, V}(G)$, the product $a\times m$ belongs to $J_{H,\, V}(K\times G)$. By Proposition \ref{catmod}, the product $a\times m$ is given by \mbox{$S_{H,\, V}(kB(K\times \ovder{G})a)(m)$},  
but $Ker (k$lin$_{\rac})$ is a subfunctor of $kB$, so $kB(K\times \ovder{G})a$ is in $Ker(k$lin$_{\rac\,K\times G\times G})$. 
Then, if $\psi$ is any element in $kB(H\times K\times G)$,  by Lemma \ref{mult},  the composition 
\begin{displaymath}
 x_i=\psi\circ (A(K\times \ovder{G})a)\circ \varphi_i
\end{displaymath}
belongs to $Ker (k$lin$_{\rac\,H\times H})$. Finally, since $V$ is a $kB(H\times H)$-module that is also a \mbox{$kR_{\rac}(H\times H)$}-module,  we must have $x_i\cdot n_i=0$ for all $i$, and hence $a\times m$ is in $J_{H,\, V}(K\times G)$.

This proves that $Ker (k$lin$_{\rac})$ acts in the trivial way on $S_{H,\, V}$ and so, this is a $kR_{\rac}$-module.
\end{proof}

The following result tells us that in order to find the simple biset functors of this lemma, it suffices to consider cyclic groups.

\begin{lema}
 \label{coci} 
$ $
\begin{itemize} 
 \item[i)] Let $G$ be any non trivial group. Then $I_{k R_{\rac}}(G)$ 
%$\sum_{\substack{|K|<|G|}}k R_{\rac}(G\times K)\circ k R_{\rac}(K\times G)$  
is equal to
\begin{displaymath}
 \sum_{\substack{K\ \textrm{cyclic}\\ \textrm{$|K|$ proper divisor of $|G|$}}}k R_{\rac}(G\times K)\circ k R_{\rac}(K\times G). 
\end{displaymath}
\item[ii)] If $G$ is not a cyclic group, then $\hat{kR}_{\rac}(G)=0$.  
\end{itemize}
\end{lema}
\begin{proof}
Let $L$ be any group. By the Artin Induction Theorem, every element $x$ in $kR_{\rac}(G\times L)$ can be written as  
\begin{displaymath}
 \sum_{\substack{C\leqslant G\times L\\ C\ \textrm{cyclic}}} a_C [1_C\ind{C}{G\times L}]
\end{displaymath}
for some $a_C\in k$. Now, in Bouc's decomposition for $(G\times L)/C$
\begin{displaymath}
 \textrm{Ind}_D^G\times_D \textrm{Inf}_{D/E}^D\times_{D/E} \textrm{Iso}_{B/A}^{D/E}\times_{B/A} \textrm{Def}_{B/A}^B\times_B \textrm{Res}_B^L
\end{displaymath}
we have that $D/E$ is a cyclic subquotient of $G$. On the other hand, since the linearization morphism is a  morphism of Green biset functors, from Lemma \ref{mult} we have
\begin{displaymath}
[1_C\ind{C}{G\times L}]= [1_X\ind{X}{G\times D_C}]\circ[1_Y\ind{Y}{D_C\times L}] 
\end{displaymath}
if we write $D_C$ for $D/E$, and suppose that $\textrm{Ind}_D^G\times_D \textrm{Inf}_{D/E}^D$ is isomorphic to $(G\times D_C)/X$ and $\textrm{Iso}_{B/A}^{D/E}\times_{B/A} \textrm{Def}_{B/A}^B\times_B \textrm{Res}_B^L$ is isomorphic to $(D_C\times L)/Y$.

To prove i) suppose that $L$ is a group of order smaller than $|G|$ and take $y$ in $k R_{\rac}(L\times G)$. If in the previous decomposition we write $y_C=[1_Y\ind{Y}{D_C\times L}]\circ y$, then $x\circ y$ is equal to
\begin{displaymath}
 \sum_{\substack{C\leqslant G\times L\\ C\ \textrm{cyclic}}} a_C [1_X\ind{X}{G\times D_C}]\circ y_C,
\end{displaymath}
which belongs to $\sum_{K}k R_{\rac}(G\times K)\circ k R_{\rac}(K\times G)$ with $K$ cyclic of order a proper divisor of $|G|$.

To prove ii) suppose that $G$ is not cyclic and take $L=G$. In this case we must have $D_C$ of order smaller than $|G|$ for every $C$, since $G$ is not cyclic. This proves that $\hat{kR}_{\rac}(G)=0$.
\end{proof}

In the rest of the section, $H$ will be a non trivial cyclic group. We will now be interested in finding which simple $kOut(H)$-modules are also $\hat{kR}_{\rac}(H)$-modules.

As we said, $\hat{kB}(H)$ is isomorphic to $k Out(H)$. This isomorphism is given by associating to each automorphism $\sigma$ of $H$, the $H\times H$-set $X_{\sigma}=(H\times H)/\Delta_{\sigma}(H)$ where
\begin{displaymath}
 \Delta_{\sigma}(H)=\{(a,\, \sigma (a))\mid a\in H\}.
\end{displaymath}
The inner automorphisms are all identified  with $X_{id}$.  
Observe that if $H$ is cyclic, then Out$(H)=$ Aut$(H)$ and  $\Delta_{\sigma}(H)$ is cyclic. 

Recall that for an abelian group $G$, the dimension of $k R_{\rac}(G)$ over $k$ is the number of cyclic subgroups of $G$. Besides, the elements of the form $[1_C\!\uparrow_C^G]$ with $C$ cyclic generate $kR_{\rac}(G)$, and so they are a basis. We shall denote this basis by $\beta(G)$.

\begin{lema}
 \label{prenuc}
Let $K$ be a cyclic group with $|K|$ a proper divisor of $|H|$. Let $\sigma$ be an automorphism of $H$. Consider a $\rac (H\times H)$-module $T$ of the form
\begin{displaymath}
 1_C\ind{C}{H\times K}\otimes_{\rac K}1_D\ind{D}{K\times H}
\end{displaymath}
with $C\leqslant H\times K$ and $D\leqslant K\times H$ cyclic subgroups.
Then the coefficient of $[1_{\Delta_{\sigma}(H)}\ind{\Delta_{\sigma}(H)}{H\times H}]$ for $[T]$ in terms of the basis $\beta(H\times H)$ is different from $0$ if and only if
%\newpage
\begin{itemize}
 \item [i)] $\pi_H(C)=H=\pi_H(D)$, with $\pi_H(C)$ and $\pi_H(D)$ being the projections on $H$ of $C$ and $D$ respectively.
 \item[ii)] $C=<(h,\, x)>\ \Longleftrightarrow\ D=<(x,\, \sigma (h))>$.
\end{itemize}
In this case, the coefficient is $|K|/|H|$.
\end{lema}
\begin{proof}
Let $\tau$ be the character associated with $T$. By Theorem 15.4 in Curtis and Reiner \cite{curtis}, the coefficients of $\tau$ in terms of the basis $\beta(H\times H)$ are given in the following way: For $G\leqslant H\times H$ cyclic, we have
\begin{displaymath}
 q_G=\frac{1}{[H\times H:G]}\sum_{\substack{G\leqslant G^*}}\mu ([G^*:G])\tau (z^*)
\end{displaymath}
where $\{G^*\}$ runs over all cyclic subgroups containing $G$, $\mu$ is the Möbius function and $z^*$ is a generator of $G^*$.

Since the exponent of $(H\times H)$ is $|H|$, then $\Delta_{\sigma}(H)$ is a maximal cyclic subgroup, that is, it is not properly contained in any other cyclic subgroup. Hence, if  $q_{\sigma}$ denotes $q_{\Delta_{\sigma}(H)}$, then 
$q_{\sigma}=(1/|H|)\tau(z)$ for $z$ a generator of $\Delta_{\sigma}(H)$.

Suppose that $H=<a>$. Let $\tau_C$ and $\tau_D$ be the characters of $1_C\ind{C}{H\times K}$ and $1_D\ind{D}{K\times H}$, respectively. By Lemma 7.1.3 in Bouc \cite{biset}, we have
\begin{displaymath}
 \tau (a,\, \sigma (a))= \frac{1}{|K|}\sum_{x\in K}\tau_C(a,\, x)\tau_D(x,\, \sigma (a)).
\end{displaymath}
So, $q_{\sigma}$ is different from $0$ if and only if there exists $x\in K$ such that $(a,\, x)$ belongs to  $C$ and $(x,\, \sigma(a))$ belongs to  $D$, and in this case
\begin{displaymath}
 \tau_C(a,\, x)=[H\times K:C]\quad \textrm{and}\quad \tau_D(x,\, \sigma (a))=[K\times H:D].
\end{displaymath}
Now, since $|K|$ is a proper divisor of $|H|$, then $(a,\, x)$ and $(x,\,\sigma (a))$ are elements of order $|H|$, and we must have $<(a,\, x)>=C$ and $<(x,\, \sigma(a))>=D$. For the same raison, $x$ is the only element of $K$ such that $(a,\, x)$ is in $C$ (respectively $(x,\, \sigma (a))$ is in $D$). From this we have $q_{\sigma}=|K|/|H|$.
\end{proof}

If $\sigma$ is an automorphism of $H$, we will write $1_{\sigma}$ for the trivial $\rac{\Delta_{\sigma}(H)}$-module. Suppose that $|H|=m$. We take the smallest non negative representatives of
$(\zze/m\zze)^{\times}$ and
 denote by $\sigma_t$ the automorphism of $H$ corresponding to the class of $t$ in $(\zze/m\zze)^{\times}$. By composition with the isomorphism between $\hat{kB}(H)$ and $k(\zze/m\zze)^{\times}$, we can consider the following extension of the linearization morphism 
\begin{displaymath}
 \ell_H:k(\zze/m\zze)^{\times}\longrightarrow \hat{kR}_{\rac}(H),
\end{displaymath}
which sends $[t]$ to the class of $[1_{\sigma_t}\ind{\Delta_{\sigma_t}(H)}{H\times H}]$.

\begin{prop}
\label{elnuc}
Let $x_n$ be
\begin{displaymath}
\sum_{[t]\in Ker\pi_{m,\,n}}[t].
\end{displaymath}
The kernel of $\ell_H$ is the ideal generated by $\{x_n\mid n\, \textrm{proper divisor of $m$}\}$.
\end{prop}
\begin{proof}
First suppose $n$ is a proper divisor of $m$, let us see that $x_n$ is in the kernel of  $\ell_H$. Let $K$ be a cyclic group of order $n$ and suppose that $H=<a>$ and $K=<x>$. Let $C=<(a,\, x)>$ and $D=<(x,\,a)>$. If $[t]$ is in the kernel of $\pi_{m,\, n}$, then
\begin{displaymath}
<(x,\, a)>=<(x,\, a)^t>=<(x,\, \sigma_t(a))>.
\end{displaymath}
So, by the previous lemma, the isomorphism class of
\begin{displaymath}
\sum_{\substack{[t]\in \textrm{}Ker \pi_{m,\, n}}}1_{\sigma_t}\ind{\Delta_{\sigma_t}(H)}{H\times H}
\end{displaymath}
appears with coefficient $|K|/|H|$ in the decomposition of
\begin{displaymath}
[1_C\ind{C}{H\times K}\otimes_{\rac K}1_D\ind{D}{K\times H}]
\end{displaymath}
in terms of $\beta(H\times H)$. Clearly, this element is $0$ in $\hat{kR}_{\rac}(H)$. Besides, if any other basic element of the form $[1_{\sigma_r}\ind{\Delta_{\sigma_r}(H)}{H\times H}]$ appears in this decomposition, then there exists an integer $d$ such that $(x,\, \sigma_r (a))^d=(x,\, a)$, but it is easy to see that this implies $[r]\in Ker\pi_{m,\, n}$.  This means that the rest of the  elements  appearing in this decomposition come from $(H\times H)$-sets which are $0$ in $\hat{kB}(H)$. Hence, $\ell_H(x_n)$ is $0$ in $\hat{kR}_{\rac}(H)$. 

Suppose now that
\begin{displaymath}
 \sum_{i=1}^da_i[s_i]
\end{displaymath}
with $[s_i] \in (\zze/m\zze)^{\times}$ and $a_i \in k$ is in the kernel of $\ell_H$. Using Lemma \ref{coci} and the Artin Induction Theorem, this means that 
\begin{equation}
 \sum_{i=1}^da_i[1_{\sigma_{s_i}}\ind{\Delta_{\sigma_{s_i}}(H)}{H\times H}] 
\end{equation}
can be written as
\begin{equation}
 \sum_{j=1}^e b_j[1_{C_j}\ind{C_j}{H\times K_j}]\circ[1_{D_j}\ind{D_j}{K_j\times H}],
\end{equation}
with $C_j$, $D_j$ and $K_j$ cyclic groups, and  $|K_j|$ a proper divisor of $|H|$.  By Lemma \ref{prenuc}, if $[1_{\sigma_{s_i}}\ind{\Delta_{\sigma_{s_i}}(H)}{H\times H}]$ is  in the decomposition of
$[1_{C_j}\ind{C_j}{H\times K_j}]\circ[1_{D_j}\ind{D_j}{K_j\times H}]$, we must have
$C_j=<(a,\, x)>$ and $D_j=<(x,\,\sigma_{s_i}(a))>$ for some $x\in K_j$. Now, if $n_j$ is the order of $x$, then, such as we did before, the isomorphism class of
\begin{equation}
\sum_{\substack{[t]\in \textrm{}Ker \pi_{m,\, n_j}}}1_{\sigma_t\sigma_{s_i}}\ind{\Delta_{\sigma_t\sigma_{s_i}}(H)}{H\times H}
\end{equation}
appears with coefficient $|K_j|/|H|$ in this decomposition. This means that if  any element in this sum has coefficient different from $0$ in (2), then the whole sum has this same coefficient. Since all the elements in (1) and (3) are basic elements, then all the elements of (1) must be of the way described in (3), which proves the claim.
\end{proof}

Note that if $Ker\pi_{m,\, n}=1$ for some $n$ proper divisor of $m$, then \mbox{$[1_{\Delta (H)}\ind{\Delta (H)}{H\times H}]$} is $0$ in $\hat{kR}_{\rac}(H)$, and hence $\hat{kR}_{\rac}(H)=0$.

\begin{coro}
 \label{simret}
Let $H$ be a group such that $\hat{kR}_{\rac}(H)\neq 0$, and $V$ a simple $k Out(H)$-module. Then $V$ is a $\hat{kR}_{\rac}(H)$-module if and only if $V$ is primitive.
\end{coro}
\begin{proof}
We know that $H$ must be cyclic. Suppose that it has order $m>1$ and let $n$ be a proper divisor of $m$. 

Suppose that $V$ is a simple $k(\zze /m\zze)^{\times}$-module that is also a $\hat{kR}_{\rac}(H)$-module. If $Ker\,\pi_{m,\, n}$ acts trivially on $V$, then there exists $v\neq 0$ in $V$ such that $x\cdot v=v$ for all $x\in Ker\,\pi_{m,\, n}$, but this implies $x_n\cdot v\neq 0$, a contradiction. 

Now suppose that $V$ is a simple primitive $k(\zze /m\zze)^{\times}$-module. If there exists $v\in V$ such that $x_n\cdot v=w$ with $w\neq 0$, then for all $x\in Ker\pi_{m,\, n}$ we have
\begin{displaymath}
 x\cdot w=x\cdot (x_n\cdot v)=xx_n\cdot v=x_n\cdot v=w
\end{displaymath}
since $xx_n=x_n$. But this is a contradiction.
\end{proof}

\section{Other examples of the conjecture}

In this section we will prove that $\com R_{\com}$ and $RB_C$ for $C$ of prime order, also satisfy Conjecture \ref{laconj}. In order to do this, we will first prove that if $A$ is a Green biset functor such that for any simple $A$-module its minimal groups are isomorphic, then $A$ satisfies Conjecture \ref{laconj}. Then we will prove that this is the case of 
$\com R_{\com}$ and $RB_C$ for $C$ of prime order.

\begin{nota}
\label{notaiso}
Observe that if $G$ and $H$ are isomorphic groups in the category $grp$, then they are also isomorphic in $\cp_A$ for any Green functor $A$. 
To prove it, suppose that $\varphi :H\rightarrow G$ is a group isomorphism. Consider $X=\biso(\varphi)$ and  $Y=\biso(\varphi^{-1})$, then define $\alpha_1=A(\ovder{X})(\lae_A)$ and $\alpha_2=A(\ovder{Y})(\lae_A)$.  It is not hard to prove that $\alpha_1\circ \alpha_2=A(\ovder{G})(\lae_A)$ and that $\alpha_2\circ \alpha_1=A(\ovder{H})(\lae_A)$. 

 If $W$ is an $A(G\times G)$-module, we denote by $^{\varphi}W$ the $A(H\times H)$-module $W$ with action of $a\in A(H\times H)$ on $w\in W$ given by $a\cdot w:=(\alpha_1\circ a\circ \alpha_2)w$.
\end{nota}

The following proposition will provide us two other examples of Green biset functors that satisfy the conjecture.

\begin{prop}
\label{isosim}
 Suppose that $A$ is a Green biset functor for which the minimal groups of each simple $A$-module form a single isomorphism class. Let $\mathcal{S}_A$ be the set of equivalence classes of couples $(H,\, V)$, where $\hat{A}(H)\neq 0$, $V$ is a simple $\hat{A}(H)$-module and $(H,\, V)\sim (G,\, W)$ if there exists an isomorphism of groups $\varphi:H\rightarrow G$ such that $V\cong {}^{\varphi}W$. 

Then the isomorphism classes of simple $A$-modules are in one-to-one correspondence with the elements of $\mathcal{S}_A$. 
\end{prop}
\begin{proof}
If $S$ is a simple $A$-module and $H$ is a minimal group for $S$, then it is minimal for any other $A$-module isomorphic to $S$. Also, as we have seen before, in this case $\hat{A}(H)\neq 0$ and  $S(H)$ is a simple $\hat{A}(H)$-module. To $S$ we associate the pair $(H,\, S(H))$. The hypothesis over the isomorphism class of $H$, justifies that this choice is well defined, because if $G$ is a group isomorphic to $H$ then $S(H)$ is isomorphic to $S(G)$ as stated in the assertion. On the other hand, if $(H,\, V)$ is a couple satisfying the hypothesis, then we will assign to it the module $S_{H,\, V}^A$ defined as the quotient $L^A_{H,\, V}/J^A_{H,\, V}$, where
\begin{displaymath}
L^A_{H,\, V}(K)=A(K\times H)\otimes_{A(H\times H)}V,\quad L_{H,\, V}(a)(x\otimes v)=(a\circ x)\otimes v
\end{displaymath}
for $a\in A(G\times K)$, and
\begin{displaymath}
 J^A_{H,\, V}(K)=\Big\{\sum_{i=1}^nx_i\otimes n_i\mid 
 \sum_{i=1}^n(y\circ x_i)\cdot n_i=0\ \forall y \in A(H\times K) \Big\}.
\end{displaymath}
Since $V$ is a simple $A(H\times H)$-module, then by Lemma 1 in Bouc \cite{ensemble}, $L^A_{H,\, V}$ has a unique simple quotient, this is $S_{H,\, V}^A$. This quotient satisfies $S_{H,\, V}^A(H)=V$.

Let us prove that if $(H,\, V)\sim (G,\, W)$, then $S_{H,\, V}^A\cong S_{G,\, W}^A$. Let $\varphi:H\rightarrow G$ be the group isomorphism which makes $V$ and ${}^{\varphi}W$
isomorphic as $A(H\times H)$-modules, then 
$L^A_{H,\, V}\cong L^A_{H,\, {}^{\varphi}W}$. We prove now that
$L^A_{H,\, {}^{\varphi}W}\cong L^A_{G,\, W}$.

Let $K\in \zz$, define
\begin{displaymath}
\mu_K : A(K\times G)\times W\longrightarrow A(K\times H)\otimes_{A(H\times H)} {}^{\varphi}W,
\end{displaymath}
by $\mu_K(a,\,  v)=a\circ \alpha_1\otimes_{A(H\times H)}v$, where $\alpha_1$ is defined as in the previous notation. Observe that $\mu_K$ can be extended to $A(K\times G)\otimes_{A(G\times G)} W$, since if  $b$ is in $A(G\times G)$, then
\begin{eqnarray*}
 \mu_K(a,\, bv) & = & a\circ \alpha_1\otimes_{A(H\times H)}bv\\
                & = & a\circ \alpha_1\otimes_{A(H\times H)}(\alpha_1\circ \alpha_2\circ b\circ \alpha_1\circ \alpha_2)v
\end{eqnarray*}
where $\alpha_2$ is defined as in the previous notation.  
Then
\begin{eqnarray*}
 \mu_K(a,\, bv) & = & a\circ \alpha_1\otimes_{A(H\times H)}(\alpha_2\circ b\circ\alpha_1)\cdot v\\
                & = & a\circ (\alpha_1\circ \alpha_2\circ b\circ\alpha_1)\otimes_{A(H\times H)}v \\
                & = & a\circ b\circ \alpha_1\otimes_{A(H\times H)}v\\
                & = & \mu_K(a\circ b,\, v).
\end{eqnarray*}
Clearly $\mu$ is natural in $K$. The inverse of $\mu$ is defined in a similar way. So $L^A_{H,\, V}$ is isomorphic to $L^A_{G,\, W}$ and hence, their unique simple quotients $S_{H,\, V}^A$ and $S_{G,\, W}^A$ are isomorphic.

Finally, we prove these assignments define a bijection.

The functor from $A(H\times H)$-mod to $\mathcal{P}_A $ that sends $V$ to $L^A_{H,\, V}$ is left adjoint of the evaluation $\mathcal{P}_A\rightarrow A(H\times H)$-mod sending each functor $M$ to $M(H)$. With this we can prove that if  $S$ is a simple $A$-module and $S(H)\neq 0$, then $S$ is isomorphic to $S_{H,\, V}^A$ with $V=S(H)$. Since $V$ is a simple $A(H\times H)$-module, then
\begin{displaymath}
 Hom_{\mathcal{P}_A}(L^A_{H,\, V},\, S)\cong Hom_{A(H\times H)}(V,\, S(H))
\end{displaymath}
implies that $S$ is a simple quotient of $L^A_{H,\, V}$ and hence it is isomorphic to $S_{H,\, V}^A$.

Now take a couple $(H,\, V)$ with $\hat{A}(H)\neq 0$ and $V$ a simple $\hat{A}(H)$-module. We see that $H$ is minimal for $S_{H,\, V}^A$: Let $K$ be a group of smaller order than $H$. Since $V$ is a $\hat{A}(H)$-module, from the definition of $S_{H,\, V}^A$, we have that $S_{H,\, V}^A(K)=0$.
\end{proof}

\subsection*{$\com R_{\com}$-modules}

\begin{prop}
\label{CRC}
 $\com R_{\com}$ is the only simple $\com R_{\com}$-module. In particular, it is a simple Green biset functor. 
\end{prop}
\begin{proof} 
By Theorem 10.33 in Curtis and Reiner \cite{curtis}, for any $G$ and $K$ groups
\begin{displaymath}
\com R_{\com}(G\times K)\cong\com R_{\com}(G)\otimes_{\com}\com R_{\com}(K) 
\end{displaymath}
which, according to Lemma \ref{prodmod}, is equivalent to
\begin{displaymath}
 \com R_{\com}(G\times K)\cong\com R_{\com}(G\times 1)\circ\com R_{\com}(1\times K).
\end{displaymath}
This implies that if $M$ is a $\com R_{\com}$-module different from $0$, then $M(1)\neq 0$.
So, by Proposition \ref{isosim}, the only simple module corresponds to $(1,\, \com)$.
Now, if $V\leqslant \com R_{\com}$ is a left ideal different from $0$, then $V(1)\neq 0$, and hence $\lae_{\com R_{\com}}=\com \in V(1)$. From this we conclude $V\cong \com R_{\com}$. 
\end{proof}

In Chapter 7 of \cite{biset}, Bouc proves that $\com R_{\com}$ is semisimple as a biset functor, its  components are the simple $\com R_{\rac}$-modules. We know that $\com R_{\rac}\cong S_{1,\,\com}$ is not a $\com R_{\com}$-module. This proposition tells us that, in fact, none of these components is a $\com R_{\com}$-module.

\subsection*{$RB_C$-modules}

The following results concern the Yoneda-Dress construction at $C$ of $RB$, for a group $C$. It is denoted by $RB_C$. If $F$ is a biset functor, then $F_C$ sends each group G to $F(G\times C)$. In an element $\varphi\in RB(G\times H)$ it is defined as $F(\varphi \times C)$. More results on this construction can be found in Section 8.2 of \cite{biset}.

\begin{lema}
\label{AH}
 If $A$ is a Green biset functor and $C$ is a group, then $A_C$ is a Green biset functor.
 \end{lema}
\begin{proof}
That $A_C$ is a biset functor is Lemma 8.2.2 in \cite{biset}. 

Let $K$ and $G$ be groups. Define $D=\{(c,\, g,\, c)\mid c\in C,\, g\in G\}$ and let $f$ be the natural isomorphism from $D$ to $G\times C$. The product for $A_C$ is given by the following composition
\[\xymatrix{
A_C(K)\times A_C(G)\ar[r]^-{\times_A} & A(K\times C\times G\times C)\ar[rr]^-{A(K\times X_C^G)} && A_C(K\times G)
}\]
where $X_C^G=\biso(f)\times_D \bres_D^{C\times G\times C}$ and $\times_A$ is the product of $A$.
The identity element of $A_C$ is $\varepsilon_{A_C}=A(\binf_{1}^{1\times C})(\varepsilon_A)\in A_C(1)$.

Let us see that this composition
defines a product in $A_C$.

Associativity:

It is enough to prove that, for groups $K$, $G$ and $L$, the exterior square in the following diagram is commutative, considering the corresponding group isomorphisms in vertices in the middle and bottom right:

\footnotesize{\[
  \xymatrix{\ar @{} [dr] |{a} 
    A_C(K)\times A_C(G)\times A_C(L)\ar[d]_{(id,\, \times_A)}\ar[r]^{(\times_A,\, id)} & \ar @{} [dr] |{b} A_C(K\times C\times G)\times A_C(L)\ar[d]_{\times_A}\ar[r]^-{(A(\alpha), id)} & A_C(K\times G)\times A_C(L)\ar[d]_{\times_A}\\  
    \ar @{} [dr] |{c\quad } A_C(K)\times A_C(G\times C\times L)\ar[d]_{(id,\, A(G\times X_C^L))}\ar[r]^-{\times_A} & \ar @{} [dr] |{d} A_C(K\times C\times  G\times C\times L)\ar[d]_{A(K\times C\times G\times X_C^L)}\ar[r]^-{A(\beta)}
    & A_C(K\times G\times C\times L)\ar[d]_{A(K\times G\times X_C^L)}\\
    A_C(K)\times A_C(G\times L)\ar[r]_{\times_A} & A_C(K\times C\times G\times L)\ar[r]_{A(K\times X_C^{G\times L})} & A_C(K\times G\times L)
}
\]}\normalsize
with $\alpha=K\times X_C^G$ and $\beta=\alpha \times L\times C$.
The square $a$ is commutative since the product of $A$ is associative. Squares $b$ and $c$ commute because the product of $A$ is natural with respect to bisets. Finally, it is not hard to see that the composition of bisets in square $d$ are isomorphic, 
and so square $d$ is commutative.

Identity element:

If $a$ belongs to $A_C(G)$, then $\varepsilon_{A_C}\times a=A(1\times X_C^G)(A(\binf_1^{1\times C})(\lae_A)\times_A a)$. Now, $A(\binf_1^{1\times C})(\lae_A)\times_A a=A(\binf_1^{1\times C}\times G\times C)(\lae_A\times_A a)$, but it is straightforward to prove that $(1\times X_C^G)\times_{1\times C\times G\times C}(\binf_1^{1\times C}\times G\times C)$ is isomorphic to the identity biset $1\times G\times C$.

Functoriality with respect to bisets:

Let $Z$ be a $(G,\, L)$-biset and $Y$ be a $(K,\, D)$-biset. We should prove the commutativity of the exterior square in
\[
 \xymatrix{
A_C(L)\times A_C(D)\ar[d]_{(A_C(Z),\, A_C(Y))}\ar[r] & A_C(L\times C\times D)\ar[d]_{A(Z\times C\times Y\times C)}\ar[r] & A_C(L\times D)\ar[d]_{A_C(Z\times Y)}\\
A_C(G)\times A_C(K)\ar[r] & A_C(G\times C\times K)\ar[r] & A_C(G\times K)
},
\]
where in the rows we have the product $\times $ of $A_C$. The square on the left commutes since the product of $A$ is natural with respect to bisets. To verify the commutativity of the square on the right, we must only verify that the bisets involved are isomorphic.
\end{proof}

Following the ideas of Example \ref{bovbur}, it is not hard to see that with this product, the composition in $\cp_{RB_C}$, which we will denote by $\times^d$ is the following. Let $X$ be a $(K\times G\times C)$-set, and $Y$ be a $(G\times L\times C)$-set. Then we can consider
$X$ as a $(K,\, G)$-biset endowed with an action of $C$, which we will suppose on the right and which commutes with those of $K$  and $G$. In the same way, $Y$ is a $(G,\, L)$-biset with an action of $C$ that commutes with those of $G$ and $L$. Hence  $X\times_G^d Y$ is the $(K,\, L)$- biset $X\times_G Y$ with a diagonal action of $C$,
\begin{displaymath}
[x,\, y]c=[xc,\, yc].
\end{displaymath}

The identity in $RB_C(G\times G)$ is $RB_C(\ovder{G})(\{\bullet \})$ which, by Bouc's  decomposition for $\ovder{G}$, equals
$(G\times G\times C)/(\Delta (G)\times C)$.
%That is, it is the $(G,\, G)$-biset $G$ with a trivial action of $C$.

Let $E$ be a subgroup of $G\times L\times C$ and $D$ be a subgroup of $L\times K\times C$, we will use the Notation 2.3.19 of \cite{biset}, $E*D$ to denote
\begin{displaymath}
 \{(g,\, k,\, c)\in G\times K\times C\mid \exists l\in L \textrm{ s. t. } (g,\, l,\, c)\in E \textrm{ and } (l,\, k,\, c)\in D\}.
\end{displaymath}

With this notation, we have a sort of a Mackey formula for 3-sets.

\begin{lema}
 \label{macbh}
Let $G$, $L$, $K$ and $C$ be groups. If $E$ is a subgroup of $G\times L\times C$ and $D$ is a subgroup of $L\times K\times C$, then we have an isomorphism of $G\times K\times C$-sets
\begin{displaymath}
 \big((G\times L\times C)/E\big)\times_L^d\big((L\times K\times C)/D\big)\cong \bigsqcup_{(l,\, c)} (G\times K\times C)/(E*{}^{(l,\, 1,\, c)}D)
\end{displaymath}
where $(l,\, c)$ is running through a set of representatives of the double cosets \mbox{$p_{2,\, 3}(E)\!\setminus\! L\times C/p_{1,\, 3}(D)$,} with $p_{2,\, 3}(E)$ and $p_{1,\, 3}(D)$ being the projections over $L\times C$ of $E$ and $D$ respectively. 
\end{lema}
\begin{proof}
If $[(g,\, l,\, c)E,\, (l',\, k,\, c')D]$ is an element on the left, then it is easy to see that its orbit under the action of $G\times K\times C$ is the same as the orbit of $[(1,\, 1,\, 1)E,\, (l^{-1}l',\, 1,\, c^{-1}c')D]$. With this observation, the proof that there exists a bijection between the orbits of $(G\times L\times C)/E\times_L^d(L\times K\times C)/D$ under the action of $G\times K\times C$ and 
$p_{2,\, 3}(E)\!\setminus\! L\times C \,/\, p_{1,\, 3}(D)$, is analogous to the proof of  Lemma 2.3.24 in Bouc \cite{biset}.
Similarly, the stabilizer of $[(1,\, 1,\, 1)E,\, (l,\, 1,\, c)D]$ equals $E*{}^{(l,\, 1,\, c)}D$.
\end{proof}

We will proceed now to prove:

\begin{prop}
\label{rbcsim}
If $C$ is a group of prime order, then the simple $RB_C$-modules are in one-to-one correspondence with the equivalence classes of couples $(H,\, V)$ where $H$ is a group %such that $\hat{RB_C}(H)\neq 0$ 
and $V$ is a simple $\hat{RB_C}(H)$-module.
\end{prop}

The proof will be given by Corollary \ref{factor} and Lemma \ref{hatRB}. The first will tell us that $RB_C$ with $C$ of primer order satisfies the hypothesis of Proposition \ref{isosim}, that is, every simple $RB_C$-module has isomorphic minimal groups. It is not hard to see that this hypothesis on a functor $A$ is fulfilled if $A$ has the following property: Suppose $G$ and $H$ are groups of order $n$. Let $I_n$ be the submodule of $A(G\times H)$ generated by elements of the form $a\circ b$ with $a\in A(G\times K)$, $b\in A(K\times H)$ and $K$ a group of order smaller than $n$. If $G$ and $H$ are not isomorphic, then $A(G\times H)=I_n$.  
All the functors we have considered have this property and we will prove next that $RB_{C}$ with $C$ of prime order has it.  
Nonetheless, we will also see that this property cannot be extended to an arbitrary group $C$. For  instance, if $C$ is a cyclic group of order 4, we will find two non isomorphic groups $G$ and $H$ of order 8 and an element $X\in B_C(G\times H)$ such that $X$ does not factor through a group of order smaller than $8$.

Until otherwise is stated, we may suppose $C$ is any group.

\begin{nota}
Let $D$ be a subgroup of $H\times K\times C$. We will write $p_1(D)$, $p_2(D)$ and $p_3(D)$ for the projections of $D$ in $H$, $K$ and $C$ respectively; $p_{1,\, 2}(D)$ will denote the  projection over $H\times K$, and in the same way we define the other possible combinations of indices. We write $k_1(D)$ for
$\{h\in p_1(D)\mid  (h,\, 1,\,1)\in D\}$ which is a normal subgroup of $p_1(D)$. Similarly, we define $k_2(D)$, $k_3(D)$ and $k_{i,\, j}(D)$ for all possible combinations of $i$ and $j$.
\end{nota}

\begin{lema}
If $D$ is a subgroup of $H\times K\times C$, then as $(H\times K\times C)$-sets, \mbox{$(H\times K\times C)/D$}  is isomorphic to 
\begin{itemize}
 \item[i)] $X\times^d_{D_1} Y$ where $D_1=p_1(D)/k_1(D)$,
%\begin{displaymath} 
 for some $X\in RB_C(H\times D_1)$ and %\quad \textrm{and}\quad 
 $Y\in RB_C(D_1\times K)$.
\item[ii)] $W\times^d_{D_2} Z$ where $D_2=p_2(D)/k_2(D)$,
%\begin{displaymath}
 for some $W\in RB_C(H\times D_2)$ and %\quad \textrm{and}\quad 
 $Z\in RB_C(D_2\times K)$.
\end{itemize}
\end{lema}
\begin{proof}
 The proof of point 2 in Proposition 2.3.25 of \cite{biset}, gives us the following group isomorphisms
\begin{displaymath}
 \frac{p_1(D)}{k_1(D)}\cong \frac{p_{2,\, 3}(D)}{k_{2,\, 3}(D)}\quad \textrm{and}\quad \frac{p_2(D)}{k_2(D)}\cong \frac{p_{1,\, 3}(D)}{k_{1,\, 3}(D)},
\end{displaymath}
the first one sends $ak_1(D)$ to $(b,\, c)k_{2,\, 3}(D)$ if $(a,\, b,\, c)$ is in $D$, the second one sends $bk_2(D)$ to $(a,\, c)k_{1,\, 3}(D)$.

Now let $X=(H\times D_1\times C)/U$ with
\begin{displaymath}
 U=\{(h,\, \alpha(h))\mid h\in p_1(D),\ \alpha:p_1(D)\twoheadrightarrow D_1 \}\times C.
\end{displaymath}
Let $\sigma$ be the isomorphism from $D_{2,\, 3}:=p_{2,\, 3}(D)/k_{2,\, 3}(D)$ to $D_1$, we define $Y=(D_1\times K\times C)/V$ with 
\begin{displaymath}
 V=\{(\sigma(\beta(k,\, c)),\,k,\, c)\mid (k,\, c)\in p_{2,\, 3}(D),\ \beta: p_{2,\, 3}(D)\twoheadrightarrow D_{2,\,3}\}.
\end{displaymath}
By Lemma \ref{macbh} we have
\begin{displaymath}
 X\times^d_{D_1}Y\cong \bigsqcup_{\substack{(t,\, c)\ \textrm{in}\\ [p_{2,\, 3}(U)\setminus D_1\times C/p_{1,\, 3}(V)]}} (H\times K\times C)/(U*{}^{(t,\, 1,\, c)}V).
\end{displaymath}
Since $p_{2,\,3}(U)=D_1\times C$, this union reduces to the element $(H\times K\times C)/U*V$. Now, $U*V$ is by definition
\begin{displaymath}
 \{(h,\, k,\, c)\mid \exists t\in D_1\textrm{ s. t. } (h,\, t,\, c)\in U,\, (t,\, k,\, c)\in V\},
\end{displaymath}
so $(h,\, k,\, c)$ is in $U*V$ if and only if $\alpha (h)=\sigma(\beta(k,\, c))$, which happens if and only if $(h,\, k,\, c)$ is in $D$.

Second decomposition is obtained in a similar way.
\end{proof}

\begin{coro}
\label{factor}
Let $C$ be a group of prime order and $H$ and $K$ be groups of order $n$. If there exists a transitive $(H\times K\times C)$-set $X$ that does not factor through a group of order smaller than $n$, then $G$ and $H$ are isomorphic.
\end{coro}
\begin{proof}
Let $X=(H\times K\times C)/D$. If it does not factor through a group of order smaller than $n$, then we must have $p_1(D)=H$, $p_2(D)=K$, $k_1(D)=1$ and $k_2(D)=1$.
From this we obtain that there exists a surjective morphism $\alpha:p_{2,\,3}(D)\rightarrow H$. If $C$ is a group of prime order, we have only two choices for $p_{2,\, 3}(D)$.
If $p_{2,\, 3}(D)=K\times C_1$  for $C_1\leqslant C$, then $\alpha(k,\, c)=\alpha_1(k)\alpha_2(c)$, where $\alpha_1$ is a morphism from $K$ to $H$. If $k$ is in $Ker\alpha_1$, then taking $(k,\, 1)$ in $p_{2,\, 3}(D)$ we have $\alpha(k,\, 1)=1$, hence $k\in k_2(D)=1$. So we have that $\alpha_1$ is injective and $H$ is isomorphic to $K$. If $p_{2,\, 3}(D)=\{(k,\, t(k))\mid t:K\rightarrow C\}$, then again through $\alpha$ we can define a surjective homomorphism from  $K$ to $H$ and so they are isomorphic. 
\end{proof}

This allows us to see that if $S$ is a simple $RB_C$-module for $C$ of prime order and $H$ and $K$ are two minimal groups for $S$, then they are isomorphic. First, they have the same order $n$. Also,
if $A=RB_C$ and $V=S(H)$, then $S\cong S^A_{H,\, V}$ and $S(K)\neq 0$. Hence there exists a transitive  element $X$ in $RB_C(H\times K)$ such that $S(X)\neq 0$ and in particular $X$ does not factor through a group of order smaller than $n$.

For the proof the next lemma we will use the following observations. Here $C$ is again any group.

Let $G$, $H$ and $K$ be any groups. Take $A\leqslant G\times K\times C$ and $B\leqslant K\times H\times C$. For $A*B$ we can define the following morphism
\begin{displaymath}
\rho_{A,\, B}: A*B\rightarrow (p_2(A)\cap p_1(B))/(k_2(A)\cap k_1(B))\quad (a,\, b,\, c)\mapsto t(k_2(A)\cap k_1(B)) 
\end{displaymath}
where $t$ is an element in $K$ such that $(a,\, t,\, c)\in A$ and $(t,\, b,\,c)\in B$ (that exists for each $(a,\, b,\, c)$ in $A*B$), it is clearly unique modulo $k_2(A)\cap k_1(B)$. The kernel of this morphism is
\begin{displaymath}
 N=\{(a,\, b,\, c)\in G\times H\times C\mid (a,\,1,\, c)\in A,\, (1,\, b,\, c)\in B\}.
\end{displaymath}

Suppose now that $A*B=D$ satisfies $p_1(D)=G$, $p_2(D)=H$, $k_1(D)=1$ and $k_2(D)=1$. Then there exists $\alpha:p_{2,\, 3}(D)\rightarrow G$ such that $D=\{(\alpha(h,\, c),\, h,\, c)\mid (h,\, c)\in p_{2,\, 3}(D)\}$. In this case, the kernel $N$ of $\rho_{A,\, B}$ is of the form
\begin{displaymath}
\{(\alpha(f_1(w),\, w),\, f_1(w),\, w)\mid w\in C'\leqslant C, f_1:C'\rightarrow H\}.
\end{displaymath}
To see this, take $(r,\, s,\, q)$ in $N$. First we must have $r=\alpha(s,\,q)$. Also, if there is another $s'$ such that $(\alpha(s',q),\, s',\, q)$ is in $N$, then $(1,\, s's^{-1},\, 1)$ is in $B$, but taking $(1,\, 1,\, 1)$ in $A$ we have $(1,\, s's^{-1},\, 1)$ in $D$, hence $s=s'$. So we can define a morphism $f_1$ from $C'=p_3(D)$ to $H$, and $N$ has the desired form.

\begin{lema}
\label{hatRB}
For any groups $G$ and $C$ we have $\hat{RB_C}(G)\neq 0$.
\end{lema}
\begin{proof}
Let $\theta$ be an automorphism of $G$ and $\zeta:C\rightarrow Z(G)$ 
be a morphism of groups. We define
%\begin{displaymath}
$ D_{\theta,\, \zeta}=\{(\theta(g)\zeta(c),\, g,\, c)\mid (g,\, c)\in G\times C\}$,
and prove that  $X_{\theta,\, \zeta}=(G\times G\times C)/D_{\theta,\, \zeta}$ %and $Y_{\theta,\, \eta}= (G\times G\times C)/E_{\theta,\, \eta}$
is different from $0$ in the quotient $\hat{RB_C}(G)$.

Suppose there exists $K$ of order smaller than $|G|$ such that $(G\times G\times C)/D$ is in $RB_C(G\times K)\times^d_K RB_C(K\times G)$. Then by Lemma \ref{macbh}, there exist $A\leqslant G\times K\times C$
and $B\leqslant K\times H\times C$, such that $D_{\theta,\, \zeta}=A*B$. According to the previous observations, if $N$ is the kernel of the morphism $\rho_{A,\, B}$, then $(A*B)/N$ has order $|G|[C:C_1]$ for a subgroup $C_1$ of $C$.
This is a contradiction, since $|(A*B)/N|$ must also divide $|K|$.
\end{proof}

Finally, we present the following example, showing that the property of Corollary \ref{factor} cannot be extended to an arbitrary group $C$.

\begin{ejem}
Let $C=<c>$ be a group of order 4, $G$ the quaternion group
\begin{displaymath}
<x,\, y\mid x^4=1,\ yxy^{-1}=x^{-1},\ x^2=y^2>
\end{displaymath}
and $H$ the dihedral group of order 8
\begin{displaymath}
<a,\, b\mid a^4=b^2=1,\ bab^{-1}=a^{-1}>.
\end{displaymath}
We will find $D$ a subgroup of $G\times H\times C$ such that the corresponding element in $B(G\times H\times C)$ does not factor through a group  of order smaller than 8.

Consider $T_1=<(a,\, c^2)>$ and $T_2=<(b,\, c)>$, subgroups of $H\times C$ of order 4. It is easy to observe that $T_1$ is a normal subgroup, so $T=T_1T_2$ is a subgroup of $H\times C$. Moreover, $T_1\cap T_2=\{1\}$ and if $\alpha=(a,\, c^2)$ and $\beta=(b,\, c)$, then $\beta\alpha\beta^{-1}=\alpha^{-1}$, so $T=T_1\rtimes T_2$. Finally, there is clearly a surjective morphism $\tau: T\rightarrow G$ whose kernel is the subgroup of order 2 generated by $\alpha^2\beta^2$. So, we define 
$D=\{(\tau(t),\, t)\mid t\in T\}$.

Observe that $p_1(D)=G$, $p_2(D)=H$, $k_1(D)=\{1\}$ and $k_2(D)=\{1\}$.

Suppose there exists $K$ of order smaller than 8 and $A\leqslant G\times K\times C$
and $B\leqslant K\times H\times C$, such that $D=A*B$. Consider the morphism $\rho_{A,\, B}$. From the observations before Lemma \ref{hatRB} we have that the kernel $N$ can only have order 1, 2 or 4. Since $[D:N]$ must also be smaller than 8, we conclude that $N$ has order 4 and it is generated by an element of the form $(r,\,s,\, c)\in D$. There are only four elements of this form in $D$, and none of them generates a normal subgroup, so we have a contradiction.
\end{ejem}

\section*{Acknowledgements}

I would like to thank Serge Bouc and Peter Webb for enlightening conversations, and Da\u ghan Yayl\i o\u glu for his help with GAP.

I also would like to thank the following institutions, for their financial support: 
ECOS - ANUIES - CONACYT for project ``Biconjuntos y funtores \mbox{asociados}'' and
Fundaci\'on Sof\'ia Koval\'evskaia - Sociedad Matemática Mexicana.

\bibliographystyle{plain}
\bibliography{amod2}

\end{document}